\documentclass[12pt]{article}

\usepackage{amsmath,amssymb,amsthm,amsfonts,textcomp}
\newcommand{\Real}{\mathbb R}
\newcommand{\Nat}{\mathbb N}
\newcommand{\norm}[1]{\left\Vert#1\right\Vert}
\newcommand{\abs}[1]{\left\vert#1\right\vert}

\newcommand{\bigabs}[1]{\bigl\vert#1\bigr\vert}

\renewcommand{\phi}{\varphi}
\newcommand{\eps}{\varepsilon}
\renewcommand{\ge}{\geqslant}
\renewcommand{\le}{\leqslant}
\renewcommand{\liminf}{\underline{\lim}}

\newcommand{\card}{{\rm card}}
\newtheorem{thm}{Theorem}
\newtheorem{prop}{Proposition}
\newtheorem{lm}{Lemma}
\newtheorem{rem}{Remark}
\newtheorem{cor}{Corollary}

\newcommand{\Wf}{\stackrel{o\ }{W_1^1}}
\newcommand{\W}{W_1^1}
\newcommand{\sign}{\mathop{\rm sign}\nolimits}
\newcommand{\dist}{\mathop{\rm dist}\nolimits}
\newcounter{pictureCounter}
\begin{document}

\title{On monotonicity of some functionals \\ under rearrangements}
\author{
S.V.~Bankevich
\footnote{JetBrains; Sergey.Bankevich@gmail.com}
\and
A.I.~Nazarov
\footnote{St.Petersburg Dept of Steklov Institute and St.Petersburg State University; al.il.nazarov@gmail.com}
}

\maketitle

\section{Introduction}

First, we recall the layer cake representation for a measurable function $u: [-1, 1] \to \Real_+$
(here and elsewhere $\Real_+ = [0,\infty)$).
Namely, if we set $\mathcal{A}_t: = \{x \in [-1,1]:\ u(x)> t \}$
then $u(x) = \int_0^\infty \chi_{\mathcal{A}_t} \, dt$.

We define the monotone rearrangement of a measurable set $E \subset [-1, 1]$ and the
monotone rearrangement of a non-negative function $u \in \W (-1, 1)$ as follows:
\begin{eqnarray*}
E^*: = [1 - \abs{E}, 1]; \qquad
u^*(x): = \int\limits_0^\infty \chi_{\mathcal{A}_t^*} \, dt.
\end{eqnarray*}

Under the same conditions we define the symmetric rearrangement 
(symmetrization) for sets and functions:
\begin{eqnarray*}
\overline{E} := [-\frac{\abs{E}}{2}, \frac{\abs{E}}{2}]; \qquad
\overline{u}(x) := \int\limits_0^\infty \chi_{\overline{\mathcal{A}_t}} \, dt.
\end{eqnarray*}

We denote by $\mathfrak{F}$ the set of continuous functions 
$F: \Real_+ \times \Real_+ \to \Real_+,$
which are convex and increasing with respect to the second argument.

Let us consider a functional
\begin{equation}
\label{functional}
I(\mathfrak a, u) = \int\limits_{-1}^1 F\big(u(x), \mathfrak a(x, u(x)) \abs{u'(x)}\big) \, dx,
\end{equation}
where $\mathfrak a: [-1, 1] \times \Real_+ \to \Real_+$ is a continuous function, $F \in \mathfrak{F}$.

It is well known that if $\mathfrak a \equiv const$ then the P\'olya--Szeg\"o type inequalities
\begin{eqnarray}
\label{toprove}
I(\mathfrak a, u^*) & \le & I(\mathfrak a, u), \qquad \qquad u \in \W(-1, 1);\\
\label{toproveSymm}
I(\mathfrak a, \overline{u}) & \le & I(\mathfrak a, u), \qquad \qquad u \in \Wf(-1, 1)
\end{eqnarray}
hold, see for example \cite{Kawohl} and references therein.

The inequality (\ref{toproveSymm}) and its multi-dimensional analogue
are proved in \cite{Br} provided that the function $\mathfrak a$ is even and convex 
with respect to $x$. However, the proof contains a gap,
and in fact this inequality was proved in \cite{Br} only for Lipschitz functions $u$.

Namely, while proving the inequality (\ref{toproveSymm}) for a natural class of functions,
the author of \cite{Br} approximates $u \in \Wf$ with finite integral (\ref{functional})
using piecewise linear functions $u_k$ and claims that $I(\mathfrak a, u_k) \to I(\mathfrak a, u)$.
However, this assertion is not justified and generally speaking is not true.
In 1926, M.A.~Lavrentiev proposed the first example of an integral functional
for which the infimum over the domain is strictly less than the infimum over the set of 
Lipschitz functions.
Historical overview and simple examples of ``one-dimensional'' 
functionals for which the Lavrentiev phenomenon takes place can be found e.g. in \cite{BGH}.
Note that a deep investigation of the Lavrentiev phenomenon for some classes of multidimensional 
functionals was carried out by V.V.~Zhikov (see, e.g., \cite{Zh1}, \cite{Zh2}).

In the paper \cite{ASC} the absence of the Lavrentiev phenomenon was proved for the functionals
$I(\mathfrak a, u) = \int_{-1}^1 F(u, u')$. Moreover it was shown that for every $u \in \W(-1, 1)$ 
there exists a sequence of Lipschitz functions $u_k$, such that
\begin{equation}
\label{convergence}
u_k \to u \text{ in } \W (-1, 1) \quad \text{ and } \quad I (\mathfrak a, u_k) \to I (\mathfrak a, u).
\end{equation}

We modify the proof from \cite{ASC} and prove the absence of the Lavrentiev phenomenon
for the functionals of the form (\ref{functional}).
This allows us to fill the gap in the proof from \cite{Br} in one-dimensional case.
In addition we prove that evenness and convexity of the weight is a necessary condition
for the inequality (\ref{toproveSymm}) to hold.

The bulk of our paper is devoted to the inequality (\ref{toprove}).
We find necessary and sufficient conditions on the weight $\mathfrak a$ for the inequality 
(\ref{toprove}) to hold%
\footnote{In particular, the inequality is satisfied if the weight function $\mathfrak a$ is even 
and concave in $x$.}.
Under certain additional assumptions this result was announced in \cite{DAN}.

We note also that the inequality (\ref{toprove}) was considered in \cite{Lan}
for functionals similar to (\ref{functional}) under additional constraint $u(-1) = 0$.
We obtain necessary and sufficient conditions for (\ref{toprove}) under this constraint.
(The author of \cite{Lan} assumed the weight $\mathfrak a$ decreasing in $x$.)

The article is divided into 8 sections.
In Section 2 we deduce the assumptions on the weight function $\mathfrak a$ which are necessary for the 
inequality (\ref{toprove}).
Auxiliary statements for weights satisfying necessary conditions are established in Section 3.
In Section 4 the inequality (\ref{toprove}) is proved for piecewise linear functions $u$.
In Section 5 we present the scheme for proving inequality (\ref{toprove}) for a wider class of 
functions $u$.
In Section 6 we prove inequality (\ref{toprove}), provided that the weight $\mathfrak a$ first increases, 
then decreases.
Section 7 is devoted to the proof of (\ref{toprove}) under necessary conditions only.
Finally, in the Section 8 we deal with symmetric rearrangement.
There we obtain necessary conditions on the weight and complete the proof of (\ref{toproveSymm}).

\section{The conditions necessary for the inequality (\ref{toprove})}

\begin{thm}
\label{necessary}
{\bf 1}. Let the inequality (\ref{toprove}) hold for some $F \in \mathfrak {F}$
and arbitrary piecewise linear $u$. Then the weight function $\mathfrak a$ is even with respect to the first argument,
that is $\mathfrak a(x, v) \equiv \mathfrak a(-x, v)$.

{\bf 2}. Let the inequality (\ref{toprove}) hold for arbitrary $F \in \mathfrak{F}$
and arbitrary piecewise linear $u$. Then the weight function $\mathfrak a$ satisfies
\begin{equation}
\label{almostConcave}
\mathfrak a(s, v) + \mathfrak a(t, v) \ge \mathfrak a(1 - t + s, v), \qquad -1 \le s \le t \le 1, v \in \Real_+.
\end{equation}
\end{thm}

\begin{proof}
{\bf 1.} Suppose that $\mathfrak a(x, v) \not \equiv \mathfrak a(-x, v)$.
Then there are $\bar{x} \in (-1, 1 )$ and $\bar{v} \in \Real_+$ such that
$$\mathfrak a(\bar{x}, \bar{v}) < \mathfrak a(-\bar{x}, \bar{v}).$$
Therefore, there is $\eps> 0$ such that
$$\bar{x} - \eps \le x \le \bar{x}, \bar{v} \le v \le \bar{v} + \eps \quad \Longrightarrow \quad \mathfrak a(x, v) < \mathfrak a(-x, v).$$
Now we introduce the following function:
$$
\left\{     
\begin{aligned}
u(x) &= \bar{v} + \eps, & x \in [-1,\bar{x}-\eps]\\
u(x) &= \bar{v} + \bar{x} - x, & x \in (\bar{x} - \eps, \bar{x})\\
u(x) &= \bar{v}, & x \in [\bar{x}, 1]
\end{aligned}
\right.
$$
Then $u^*(x, v) = u(-x, v)$ and
\begin{multline*}
I(\mathfrak a, u)-I(\mathfrak a, u^*) \\
= \int\limits_{\bar{x}-\eps}^{\bar{x}} F\big( \bar{v} + \bar{x} - x, \mathfrak a(x, \bar{v} + \bar{x} - x) \big) \, dx -
\int\limits_{-\bar{x}}^{-\bar{x}+\eps} F\big( \bar{v} + \bar{x} + x, \mathfrak a(x, \bar{v} + \bar{x} + x) \big) \, dx \\
= \int\limits_{\bar{x}-\eps}^{\bar{x}} \big( F\big( \bar{v} + \bar{x} - x, \mathfrak a(x, \bar{v} + \bar{x} - x) \big) -
F\big( \bar{v} + \bar{x} - x, \mathfrak a(-x, \bar{v} + \bar{x} - x) \big) \big) \, dx < 0,
\end{multline*}
which contradicts the assumption. Thus, the first statement is proved.

{\bf 2.} Suppose that the assumption (\ref{almostConcave}) is not satisfied.
Then, by continuity of $\mathfrak a$, there exist $-1 \le s \le t \le 1$, $\eps, \delta> 0$ and $\bar{v} \in \Real_+$, such that
for any $0 \le y \le \eps$ and $\bar{v} \le v \le \bar{v} + \eps$ the following inequality holds:
$$\mathfrak a(s + y, v) + \mathfrak a(t - y, v) + \delta < \mathfrak a( 1 - t + s + 2y, v).$$

Consider the function $u$ (see fig. \ref{uGraph}):
\begin{equation}
\label{parLinU}
\left\{     
\begin{aligned}
u(x) &= \bar{v}, & x \in [-1, s] \cup [t, 1]\\
u(x) &= \bar{v} + x - s, & x \in [s, s + \eps]\\
u(x) &= \bar{v} + \eps, & x \in [s + \eps, t - \eps]\\
u(x) &= \bar{v} + t - x, & x \in [t - \eps, t]
\end{aligned}
\right.
\end{equation}

\begin{center}
\begin{picture}(200,90)
\refstepcounter{pictureCounter}
\label{uGraph}
\put(10,65){\line(1,0){50}}
\put(60,65){\line(1,1){10}}
\put(70,75){\line(1,0){40}}
\put(110,75){\line(1,-1){10}}
\put(120,65){\line(1,0){70}}
\put(0,25){\vector(1,0){200}}
\put(100,15){\vector(0,1){80}}
\put(99,65){\line(1,0){2}}
\put(92,62){$\bar{v}$}
\put(60,24){\line(0,1){2}}
\put(58,14){$s$}
\put(120,24){\line(0,1){2}}
\put(119,14){$t$}
\put(10,24){\line(0,1){2}}
\put(6,14){$-1$}
\put(190,24){\line(0,1){2}}
\put(188,14){$1$}
\put(20,70){$u(x)$}
\put(85,1){Fig. \arabic{pictureCounter}}
\end{picture}
\end{center}
Then
$$
\left\{     
\begin{aligned}
u^*(x) &= \bar{v}, & x \in [-1, 1 - t + s]\\
u^*(x) &= \bar{v} + \frac{ x - ( 1 - t + s ) }{2}, & x \in [1 - t + s, 1 - t + s + 2\eps]\\
u^*(x) &= \bar{v} + \eps, & x \in [1 - t + s + 2\eps, 1]
\end{aligned}
\right.
$$
(see fig. \ref{uStarGraph}).

\begin{center}
\begin{picture}(200,90)
\refstepcounter{pictureCounter}
\label{uStarGraph}
\put(10,65){\line(1,0){120}}
\put(130,64){\line(2,1){20}}
\put(150,75){\line(1,0){40}}
\put(0,25){\vector(1,0){200}}
\put(100,15){\vector(0,1){80}}
\put(99,65){\line(1,0){2}}
\put(92,67){$\bar{v}$}
\put(130,24){\line(0,1){2}}
\put(110,14){$1 - t + s$}
\put(10,24){\line(0,1){2}}
\put(6,14){$-1$}
\put(190,24){\line(0,1){2}}
\put(188,14){$1$}
\put(20,70){$u^*(x)$}
\put(85,1){Fig. \arabic{pictureCounter}}
\end{picture}
\end{center}

We have
\begin{multline*}
I(\mathfrak a, u^*) = \int\limits_0^{2\eps} F \big( u(1 - t + s + z), \frac{\mathfrak a(1 - t + s + z, u(1 - t + s + z))}{2} \big) \, dz\\
= \int\limits_0^\eps 2 F \big(\bar{v} + y, \frac{\mathfrak a(1 - t + s + 2y, \bar{v} + y)}{2} \big) \, dy\\
0 \le I( \mathfrak a, u ) - I( \mathfrak a, u^* ) =
\int\limits_0^\eps \big( F\big(\bar{v} + y, \mathfrak a(s + y, \bar{v} + y)\big) + F\big(\bar{v} + y, \mathfrak a( t - y, \bar{v} + y)\big)\\
- 2 F \big(\bar{v} + y, \frac{ \mathfrak a(1 - t + s + 2y, \bar{v} + y) }{2} \big) \big) \, dy\\
< \int\limits_0^\eps \big( F\big(\bar{v} + y, \mathfrak a(s + y, \bar{v} + y)\big) + F\big(\bar{v} + y, \mathfrak a(t - y, \bar{v} + y)\big)\\
- 2 F \big( \bar{v} + y, \frac{ \mathfrak a(s + y, \bar{v} + y) + \mathfrak a(t - y, \bar{v} + y) + \delta }{2} \big) \big) \, dy =: J.
\end{multline*}

Let us consider the function $F(v, p) = p ^ \alpha$.
For $\alpha = 1$, the following inequality trivially holds:
\begin{equation}
\label{anticonvex}
\frac{F(v, p) + F(v, q)}{ 2 } - F\big(v, \frac{p + q}{ 2 } + \frac{\delta}{ 2 }\big) <0.
\end{equation}
We are interested in $p, q$ from the compact $[0 , A]$,
where 
\begin{equation}
\label{weightMax}
A=\max \limits_{(x, v)} \mathfrak a,\qquad (x, v) \in [-1, 1 ] \times u([-1, 1] ).
\end{equation}
Therefore, there is an $\alpha> 1$, for which the inequality (\ref{anticonvex})
still holds.
For example, any $1 < \alpha < (\log_2 \frac{ 2 A}{A + \delta})^{-1}$ is suitable.

Thus, we obtain a function $F$ strictly convex with respect to the second argument
for which $J \le 0$. This contradiction proves the second statement.
\end{proof}

\begin{rem}
\label{landesNecessary}
It can be seen that proving the second statement of Theorem \ref{necessary}
one can replace the function $u$ on the interval $[-1, s]$ by any increasing function.
Thus, in the case where $u$ is pinned at the left end {\rm ($u(-1) = 0$)}
the assumption (\ref{almostConcave}) is also necessary for the inequality (\ref{toprove}) to hold.
\end{rem}

\begin{rem}
Let $\mathfrak a(\cdot, v)$ be even.
Then the assumption (\ref{almostConcave}) is equivalent to subadditivity of the function $\mathfrak a(1 - \cdot, v)$.
In particular, if a non-negative function $\mathfrak a$ is even and concave with respect to the first argument then it satisfies the assumption (\ref{almostConcave}).
\end{rem}

\section{Properties of the weight function}

For brevity, in this section we omit the second argument of the function $\mathfrak a$.
Thus, we assume, that $\mathfrak a \in C[-1, 1]$ and $\mathfrak a \ge 0$.

\begin{lm}
\label{weightSum}
Let $\mathfrak a$ satisfy (\ref{almostConcave}).

{\bf 1.} For any $-1 \le t_1 \le t_2 \le \ldots \le t_n \le 1$
the following inequalities hold
\begin{align*}
\sum_{k = 1}^n \mathfrak a(t_k) & \ge \mathfrak a( 1 - \sum_{k = 1}^n (-1)^k t_k), & \text{ for even $n$}, & \\
\sum_{k = 1}^n \mathfrak a(t_k) & \ge \mathfrak a(- \sum_{k = 1}^n (-1)^k t_k), & \text{ for odd $n$}. &
\end{align*}

{\bf 2.}
Assume that in addition the function $\mathfrak a$ is even.
Then the following inequalities also hold:
\begin{align*}
\sum_{k = 1}^n \mathfrak a(t_k) & \ge \mathfrak a(-1 + \sum_{k = 1}^n (-1)^k t_k), & \text{ for even $n$}, & \\
\sum_{k = 1}^n \mathfrak a(t_k) & \ge \mathfrak a(\sum_{k = 1}^n (-1)^k t_k), & \text{ for odd $n$}. &
\end{align*}
\end{lm}

\begin{proof}
{\bf 1.}
We prove the lemma by induction.
For $n = 1$ the assertion is trivial.
Now let $n$ be even. Then, by the induction hypothesis,
$$\sum_{k=1}^{n - 1} \mathfrak a(t_k) \ge \mathfrak a( -\sum_{k = 1}^{n - 1} (-1)^k t_k ).$$
Then
$$\sum_{k = 1}^{n - 1} \mathfrak a( t_k ) + \mathfrak a( t_n ) \ge \mathfrak a( -\sum_{k = 1}^{n - 1} (-1)^k t_k ) + \mathfrak a( t_n ) \ge
\mathfrak a( 1 - \sum_{k = 1}^{n} (-1)^k t_k ).$$
In the case of odd $n$ we have the following induction hypothesis:
$$\sum_{k=2}^n \mathfrak a(t_k) \ge \mathfrak a( 1 + \sum_{k = 2}^n (-1)^k t_k ).$$
Then
$$\mathfrak a( t_1 ) + \sum_{k = 2}^n \mathfrak a( t_k ) \ge \mathfrak a( t_1 ) + \mathfrak a( 1 + \sum_{k = 2}^{n} (-1)^k t_k ) \ge
\mathfrak a( -\sum_{k = 2}^{n} (-1)^k t_k + t_1 ) = \mathfrak a( -\sum_{k = 1}^{n} (-1)^k t_k ).$$

{\bf 2.} The proof of this part is trivial. 
\end{proof}

\begin{lm}
\label{periodicity}
{\bf 1.} Let $\mathfrak a$ satisfy $(\ref{almostConcave})$.
If there is $x_0 \in [-1, 1]$, such that $\mathfrak a(x_0) = 0$,
then either $\mathfrak a \Big|_{[x_0, 1]} \equiv 0$
or the set of zeros of $\mathfrak a$ is periodic on $[x_0, 1]$
and the period is a divisor of $1 - x_0$.

{\bf 2.} Let $\mathfrak a$ be even and satisfy $(\ref{almostConcave})$.
If there is $x_0 \in [-1, 1]$, such that $\mathfrak a(x_0) = 0$,
then either $\mathfrak a \equiv 0$
or the function $\mathfrak a$ is periodic on $[-1, 1]$
and the period is a divisor of $1 - x_0$.
\end{lm}

\begin{proof}
{\bf 1.}
Note that if $\mathfrak a(s) = \mathfrak a(t) = 0$ for some $s \le t$
then the inequality (\ref{almostConcave}) implies
$$0 = \mathfrak a(s) + \mathfrak a(t) \ge \mathfrak a( 1 - (t - s) ) \ge 0$$
i.e. $\mathfrak a(1 - (t - s)) = 0$.
Substituting $s = t = x_0$, we obtain $\mathfrak a(1) = 0$.

Similarly, if $s \le 1 - t$ and $\mathfrak a(s) = \mathfrak a(1 - t) = 0$, then $\mathfrak a(s + t) = 0$.

Thus, the set of roots of $\mathfrak a$ is symmetric on the segment $[x_0, 1]$ and
whenever $s$ and $s + \Delta$ ($\Delta \ge 0$) are roots of $\mathfrak a$,
values $s + k\Delta$ are roots of $\mathfrak a$ too provided $s + k\Delta \le 1$.
This implies the set of roots of $\mathfrak a$ is periodic on $[x_0, 1]$
or coincides with it.

{\bf 2.} The periodicity of zeros of the function $\mathfrak a$ follows from its evenness and from the first assertion of the lemma.
Denote the distance between consecutive zeros by $\Delta$.

Then for $-1 \le x \le 1 - \Delta$ the following holds
$$\mathfrak a(x) = \mathfrak a(x) + \mathfrak a(1 - \Delta) \ge \mathfrak a(x + \Delta).$$

On the other hand, $-1 \le -(x + \Delta) \le 1 - \Delta$, and
$$\mathfrak a(x + \Delta) = \mathfrak a(-(x + \Delta)) + \mathfrak a(1 - \Delta) \ge \mathfrak a(-x) = \mathfrak a(x).$$

Thus, $\mathfrak a(x) = \mathfrak a(x + \Delta)$.
\end{proof}

\begin{lm}
\label{maxSumConcave}
Suppose that $\mathfrak a_1$ and $\mathfrak a_2$ satisfy $(\ref{almostConcave})$.
Then the functions $\max (\mathfrak a_1(x), \mathfrak a_2(x))$ and $\mathfrak a_1(x) + \mathfrak a_2(x)$ also satisfy $(\ref{almostConcave})$.
\end{lm}
\begin{proof}
Set $\mathfrak a(x) = \max (\mathfrak a_1(x), \mathfrak a_2(x))$. Then
\begin{multline*}
\mathfrak a(1 - t + s) = \max(\mathfrak a_1( 1 - t + s), \mathfrak a_2(1 - t + s)) \le
\max(\mathfrak a_1(s) + \mathfrak a_1(t), \mathfrak a_2(s) + \mathfrak a_2(t)) \\
\le \max(\mathfrak a_1(s), \mathfrak a_2(s)) + \max(\mathfrak a_1(t), \mathfrak a_2(t)) =
\mathfrak a(s) + \mathfrak a(t).
\end{multline*}

The second part is obvious.
\end{proof}

\begin{lm}
\label{piecewiseLinearConcave}
Let the function $\mathfrak a$ satisfy $(\ref{almostConcave})$, $k \in \Nat$.
Then a piecewise linear function $\mathfrak a_k$,
interpolating $\mathfrak a$ using the nodes
$(-1 + \frac{2i}{k})$, $i = 0, 1, \dots, k$,
also satisfies $(\ref{almostConcave})$.
\end{lm}
\begin{proof}
{\bf 1.}
Let $s = -1 + \frac{2i}{k}$, $t = -1 + \frac{2j}{k}$.
Then the inequality $(\ref{almostConcave})$ holds for $\mathfrak a_k$, because it does for $\mathfrak a$,
and their values at these points coincide.

{\bf 2.}
Now let $s = -1 + \frac{2i}{k}$ and $t \in [-1 + \frac{2j}{k}, -1 + \frac{2(j + 1)}{k}]$.

Consider the linear function $h_1(t) = \mathfrak a_k( 1 - t + s ) - \mathfrak a_k(t) - \mathfrak a_k(s)$.
It follows from part 1 that $h_1(-1 + \frac{2j}{k}) \le 0$ and $h_1(-1 + \frac{2(j + 1)}{k}) \le 0$.
Since $h_1$ is linear, $h_1(t) \le 0$.
Thus, the inequality holds for every $s = -1 + \frac{2i}{k}$ and $t \in [-1, 1]$.

{\bf 3.}
Let $s$ and $t$ satisfy $1 - t + s = \frac{2j}{k}$.

Consider the function $h_2(y) = \mathfrak a_k(\frac{2j}{k}) - \mathfrak a_k(s + y) - \mathfrak a_k(t + y)$.
If we choose $y_0$ such that $s + y_0$ is one of the nodes then $t + y_0$ is also a node.
Therefore, $h_2(y_0) = \mathfrak a(\frac{2j}{k}) - \mathfrak a(s + y_0) - \mathfrak a(t + y_0) \le 0$.
Since $h_2$ is linear between such $y_0$'s, we obtain $h_2(y) \le 0$ for all admissible $y$.

{\bf 4.}
Finally, consider $h_3(s) = \mathfrak a_k( 1 - t + s ) - \mathfrak a_k(t) - \mathfrak a_k(s)$ for arbitrary given $t \in [-1, 1]$.
Note that parts 2 and 3 imply $h_3(s) \le 0$ for any $s$
such that either $s$ or $1 - t + s$ is a node.
Since $h_3$ is linear between these points, $h_3(s) \le 0$ for all admissible $s$,
and the statement follows.
\end{proof}

\section{The result for piecewise linear functions}

\rm
In this section we prove the inequality (\ref{toprove}) for piecewise linear functions.
Without loss of generality, we assume that $F(\cdot, 0) \equiv 0$.

\begin{thm}
\label{linth}
Let the function $\mathfrak a$ be even and satisfy the condition $(\ref{almostConcave})$.
If $u$ is a nonnegative piecewise linear function then $I(\mathfrak a, u) \ge I(\mathfrak a, u^*)$.
\end{thm}

\begin{proof}
Let $-1 = x_1 < x_2 < \dots < x_K = 1$ be the nodes of $u$.
Consider the set $U$ equal to the range of $u$ with images of endpoints of linear pieces excluded:
$U := u( [-1, 1] ) \setminus \{ u(x_1), \dots, u(x_K) \}$.
It's obvious that the set $U$ is the union of a finite number of intervals $U = \cup_{j = 1}^N G_j$.

We denote by $m_j$ the number of preimages for $u_0 \in G_j$,
i.e. the number of solutions of the equation $u(y) = u_0$
(obviously, $m_j$ does not depend on $u_0 \in G_j$).
It is easy to see that the preimages are linear functions of $u_0$:
$y = y_k^j(u_0)$, $k = 1, \dots, m_j$,
and $y_k^j{}'(u(y)) = \frac{1}{u'(y)}$.
We assume that $y_1^j(u_0) < y_2^j(u_0) < \dots < y_{m_j}^j(u_0)$.

The solution of the equation $u^*(y^*)=u_0$ ($u_0 \in U$) can be expressed in terms of $y_k^j$:

\begin{center}
\begin{tabular}{l|l|l} 
$u(-1)<u_0$ & $m_j$ is even & $y^*=1-\sum\limits_{k=1}^{m_j} (-1)^k y_k^j$ \rule[-17pt]{0pt}{40pt} \\
            & $m_j$ is odd  & $y^*=-\sum\limits_{k=1}^{m_j} (-1)^k y_k^j$ \rule[-17pt]{0pt}{40pt} \\ \hline
$u(-1)>u_0$ & $m_j$ is even & $y^*=-1+\sum\limits_{k=1}^{m_j} (-1)^k y_k^j$ \rule[-17pt]{0pt}{40pt} \\
            & $m_j$ is odd  & $y^*=\sum\limits_{k=1}^{m_j} (-1)^k y_k^j$ \rule[-17pt]{0pt}{40pt} \\ 
\end{tabular}
\end{center}

Let $y^*(v) = (u^*)^{-1}(v)$.
Then $y^*{}'(v) = \sum_{k=1}^{m_j} \abs{y_k^j{}'(v)}$ for $v \in G_j$, as the signs in the expression for
$y^*$ and signs of $y_k^j{}'$ alternate, and $y^*{}'(v)\ge 0$.

The sets of zeros of $u'(x)$ and $u^*{}'(x)$ can have nonzero measure.
However, they do not contribute to the integral, since $F\big(u(x), 0\big) = 0$.

Consider the remaining parts of the integrals :
\begin{multline*}
I(\mathfrak a, u) = \sum_{j=1}^N \,\int\limits_{u^{-1}(G_j)} F\big(u(x), \mathfrak a(x, u(x)) \abs{u'(x)}\big) \, dx
\\ = \sum_{j=1}^N \,\int\limits_{G_j} \sum_{k=1}^{m_j} F\Big(v, \frac{\mathfrak a(y_k^j(v), v)}{\bigabs{y_k^j{}'(v)}}\Big) \bigabs{y_k^j{}'(v)} \, dv,
\end{multline*}
\begin{multline*}
I(\mathfrak a, u^*) = \sum_{j=1}^N \,\int\limits_{(u^*)^{-1}(G_j)} F\big(u^*(x), \mathfrak a(x, u(x)) \bigabs{u^*{}'(x)}\big) \, dx
\\ = \sum_{j=1}^N \,\int\limits_{G_j} F\Big(v, \frac{\mathfrak a(y^*(v), v)}{\sum_{k=1}^{m_j} \bigabs{y_k^j{}'(v)}}\Big)
\sum_{k=1}^{m_j} \bigabs{ y_k^j{}'(v) } \, dv.
\end{multline*}

We fix $j$ and $v$ in the right parts and prove the inequality for integrands.
We denote $b_k := |y_k^j{}'(v)|$, $y_k := y_k^j(v)$, $y^* := y^*(v)$, $m := m_j$.
Then the assertion takes the form:
$$T:=\sum_{k=1}^m b_k F\Big( v, \frac{ \mathfrak a(y_k, v) }{b_k} \Big)
\ge F\Big( v, \frac{ \mathfrak a(y^*, v) }{ \sum_{k=1}^m b_k  } \Big) \sum_{k=1}^m b_k.$$
By Jensen's inequality for the function $F(v, \cdot)$, we obtain
$$T \ge F\Big( v, \frac{ \sum_{k=1}^m \mathfrak a(y_k, v) }{ \sum_{k=1}^m b_k } \Big) \sum_{k=1}^m b_k.$$
Then it is sufficient to prove $\sum_{k=1}^m \mathfrak a(y_k, v) \ge \mathfrak a(y^*, v)$, which is true due to Lemma \ref{weightSum}.
\end{proof}

\begin{rem}
\label{landesLinear}
In the paper $\cite{Lan}$ the inequality $(\ref{toprove})$ is proved under the additional assumption $u(-1) = 0$
for the weight functions $\mathfrak a$, decreasing in $x$.
It is easy to see that under this assumption, the proof of Theorem $\ref{linth}$ works for weights satisfying
$(\ref{almostConcave})$ without the evenness assumption,
since in this case $u(-1) < u_0$, and we need only two of the four inequalities,
given by the first part of Lemma $\ref{weightSum}$.
It is also obvious that the assumption $(\ref{almostConcave})$ is weaker than the assumption of $\mathfrak a$ decreasing in $x$.
\end{rem}

\section{Extension of class of functions for which inequality (\ref{toprove}) holds}
The next statement is rather standard. However, we give a full proof for the reader's convenience.

\begin{lm}
Let the function $\mathfrak a$ be continuous. Then the functional $I(\mathfrak a, u)$ is weakly lower semicontinuous in $\W(-1, 1)$.
\label{lowersemi}
\end{lm}

\begin{proof}
Let $u_m \rightharpoondown u$ in $\W(-1, 1)$.
Let's denote $A = \varliminf I( \mathfrak a, u_m ) \ge 0$.
We are going to prove $I(\mathfrak a, u) \le A$.
In the case $A = \infty$ the assertion is trivial, so we can assume $A < \infty$.
Switching to a subsequence, we obtain $A = \lim I( \mathfrak a, u_m )$.

Weak convergence implies, that there exists
$R_0$ such that $\norm{ u_m }_{\W(-1, 1)} \le R_0$.
Moreover, switching to a subsequence, we can assume that $u_m \to u$ in $L_1(-1, 1)$
and $u_m(x) \to u(x)$ almost everywhere.
Then, by Egorov's theorem, for any $\eps$ there exists a set
$G_\eps^1$ such that $\abs{ G_\eps^1 } < \eps$ and $u_m \rightrightarrows u$ in $[-1, 1] \setminus G_\eps^1$.

Uniform convergence of $u_m$ implies there exists $K$ such that for each $m>K$
the inequality $\abs{u_m} \le \abs{u} + \eps$ holds in $[-1, 1] \setminus G_\eps^1$.
Let $G_\eps^2 = \{x \in [-1, 1] \setminus G_\eps^1 : \abs{u(x)} \ge \frac{R_0 + \eps}{\eps} \}$.
Then
$$R_0 \ge \int\limits_{-1}^1 \abs{ u(x) } \, dx \ge \int\limits_{G_\eps^2} \abs{ u(x) } \, dx \ge
\int\limits_{G_\eps^2} \frac{R_0 + \eps}{\eps} \, dx = \abs{G_\eps^2} \frac{R_0 + \eps}{\eps}$$
That is, $\abs{G_\eps^2} \le \eps \frac{R_0}{R_0 + \eps} < \eps$.
Thus, the functions $u_m$ converge uniformly and are uniformly bounded outside the set $G_\eps := G_\eps^1 \cup G_\eps^2$.

Continuity of $F$ and $\mathfrak a$ implies that for any $\eps$ and $R$, there exists
$N( \eps, R )$, such that if $x \in [-1, 1] \setminus G_\eps$, $\abs{ M } \le R$ and $m > N( \eps, R )$ then
$$| F\big( u_m( x ), \mathfrak a( x, u_m( x ) ) M \big) - F\big( u( x ), \mathfrak a( x, u( x ) ) M \big) | < \eps.$$

Let $E_{m,\eps} := \{ x \in [-1, 1]: \abs{ u_m'( x ) } \ge \frac{ R_0 }{ \eps } \}$.
Then
$$R_0 \ge \int\limits_{-1}^1 \abs{ u_m'( x ) } \, dx \ge \int\limits_{ E_{m,\eps} } \abs{ u_m'( x ) } \, dx \ge
\int\limits_{ E_{m,\eps} } \frac{ R_0 }{ \eps } \, dx = \frac{ R_0 }{ \eps } \abs{ E_{m,\eps} }.$$
Therefore $\abs{ E_{m,\eps} } \le \eps$.

Finally we set $L_{m,\eps} := [-1, 1] \setminus ( E_{m,\eps} \cup G_\eps )$.
Note, that $\abs{ L_{m,\eps} } \ge 2 - 3 \eps$.

We put $R := \frac{ R_0 }{ \eps }$, $N( \eps ) := N( \eps, \frac{ R_0 }{ \eps } )$.
For any $\eps > 0$, $x \in L_{m,\eps}$ and $m > N( \eps )$ we have
$$\Big | F\big( u_m( x ), \mathfrak a( x, u_m( x ) ) \abs{u_m'( x )} \big) - F\big( u( x ), \mathfrak a( x, u( x ) ) \abs{u_m'( x )} \big) \Big | < \eps,$$
thus
\begin{equation}
\label{frDer}
\int\limits_{L_{m,\eps}} \Big | F\big( u_m( x ), \mathfrak a( x, u_m( x ) ) \abs{u_m'( x )} \big) - F\big( u( x ), \mathfrak a( x, u( x ) ) \abs{u_m'( x )} \big) \Big | \, dx < 2 \eps.
\end{equation}

We put $\eps_j = \frac{ \eps }{ 2^j }$ ($j \ge 1$), $m_j = N( \eps_j ) + j \to \infty$ and $L_\eps = \bigcap L_{m_j,\eps_j}$.
Then $\sum \eps_j = \eps$ and therefore $\abs{ [-1, 1] \setminus L_\eps } < 3 \eps$.
Since (\ref{frDer}) implies
$$\int\limits_{L_\eps} \Big | F\big( u_{m_j}( x ), \mathfrak a( x, u_{m_j}( x ) ) |u_{m_j}'( x )| \big) - F\big( u( x ), \mathfrak a( x, u( x ) ) |u_{m_j}'( x )| \big) \Big | \, dx < 2 \eps_j,$$
we obtain
\begin{multline*}
A = \lim I (\mathfrak a, u_{m_j}) = \lim \int\limits_{-1}^1 F\big(u_{m_j}(x), \mathfrak a(x, u_{m_j}(x)) | u_{m_j }'(x) |\big) \, dx \\
\ge \varliminf \int\limits_{-1}^1 \chi_{L_\eps}(x) F\big(u (x), \mathfrak a(x, u(x)) | u_{m_j}'(x) |\big) \, dx
=: \varliminf J_\eps(u_{m_j}').
\end{multline*}

The functional
$$J_\eps( v ) = \int\limits_{-1}^1 \chi_{L_\eps}( x ) F\big( u( x ), \mathfrak a( x, u( x ) ) |v( x )| \big) \, dx$$
is convex.
Switching to a subsequence $u_k$ again, we can assume that
$\varliminf J_\eps( u_{m_j}' ) = \lim J_\eps( u_k' )$.
Since $u_k' \rightharpoondown u'$ in $L_1$, we can choose a sequence of convex combinations of $u_k'$,
which converges to $u'$ strongly (see \cite[Theorem 3.13]{Rudin}).
Namely, there are $\alpha_{k,l} \ge 0$ for
$k \in \Nat$, $l \le k$, such that $\sum_{l = 1}^k \alpha_{k,l} = 1$ for every $k$ and
$w_k := \sum_{l = 1}^k \alpha_{k,l} u_{l}' \to u'$ in $L_1$.
Also, without loss of generality we can assume that the minimal index $l$ of a nonzero coefficient $\alpha_{k,l}$
tends to infinity as $k$ tends to infinity.
Then
$$\lim J_\eps( u_k' ) = \lim \sum_{l = 1}^k \alpha_{k,l} J_\eps( u_{l}' ).$$

By the convexity of $J_\eps$, we have
$$\sum_{l = 1}^k \alpha_{k,l} J_\eps( u_{l}' ) \ge J_\eps( w_k ).$$

Finally, since $w_k \to u'$ in $L_1(-1, 1)$, we can assume, by switching to a subsequence, that $w_k(x) \to u'(x)$ almost everywhere.
Moreover, since $\abs{ u_j'( x ) } < \frac{ R_0 }{\eps}$ holds for $x \in L_\eps$, then $\abs{ w_k( x ) } < \frac{ R_0 }{\eps}$.
Hence,
$$F\big( u( x ), \mathfrak a( x, u( x ) ) |w_k( x )| \big) \le \max\limits_{(x, M)} F\big( u( x ), \mathfrak a( x, u( x ) ) M \big) < \infty,$$
where the maximum is taken over a compact set
$(x,M) \in [-1, 1] \times [-\frac{ R_0 }{\eps},\frac{ R_0 }{\eps}]$.
Therefore, by the Lebesgue theorem, $\lim J_\eps(w_k) = J_\eps(u')$.
Thus,
$$A \ge \lim J_\eps( u_k' ) = \lim \sum_{l = 1}^k \alpha_{k,l} J_\eps( u_{l}' ) \ge
\varliminf J_\eps( w_k ) = J_\eps( u' ).$$

Since $\eps > 0$ is arbitrary, $A \ge I(\mathfrak a, u)$ follows.
\end{proof}

\begin{lm}
\label{uplift}
Let $B \subset A \subset \W(-1,1)$.
Let the inequality (\ref{toprove}) hold for any $u \in B$.
Suppose that for each $u \in A$
there is a sequence $u_k \in B$ such that relation (\ref{convergence}) holds.
Then the inequality (\ref{toprove}) holds for any $u \in A$.
\end{lm}
\begin{proof}
Let us pick some $u \in A$ and find an appropriating sequence $\{u_k\} \subset B$.
By hypothesis, $I(\mathfrak a, u_k^*) \le I(\mathfrak a, u_k) \to I(\mathfrak a, u)$.
By \cite[Theorem 1]{Br}
$$u_k \to u \text{ in } \W(-1, 1) \quad \Longrightarrow \quad \overline{u_k} \rightharpoondown \overline{u} \text{ in } \W(-1, 1).$$
Since $u_k^*( x ) = \overline{u_k}( \frac{x - 1}{2} )$ and $u^*( x ) = \overline{u}( \frac{x - 1}{2} )$,
we have $u_k^* \rightharpoondown u^*$ in $\W(-1, 1)$.
By Lemma \ref{lowersemi}, we obtain 
$$I(\mathfrak a, u^*) \le \liminf I(\mathfrak a, u_k^*) \le \lim I(\mathfrak a, u_k) = I(\mathfrak a, u).$$
\end{proof}

\begin{cor}
Let the weight $\mathfrak a$ be continuous, and let the inequality $(\ref{toprove})$ hold for non-negative piecewise linear functions $u$.
Then it holds for all non-negative Lipschitz functions.
\end{cor}
\begin{proof}
By Theorem 1 in Section 6.6 \cite{Gariepy}, any Lipschitz function $u$ can be approximated by $u_k \in C^1[-1, 1]$ such that
$$u_k \rightrightarrows u, \qquad u_k' \to u' \text{ a.e.}, \qquad |u_k'| \le const.$$
By the Lebesgue theorem relation (\ref{convergence}) holds.
In turn, $u_k$ can be approximated in the same way by piecewise linear functions.
Using Theorem \ref{linth} and applying Lemma \ref{uplift}, we complete the proof.
\end{proof}

\section{The inequality for $u \in \W(-1, 1)$ with an additional restriction on weight}
\label{ASC}

In this section we prove the inequality (\ref{toprove}) under the additional condition:
weight is monotonic in $x$ for $x \in [-1, 0]$ and $x \in [0, 1]$.

\begin{lm}
\label{Wapprox}
Let $\mathfrak a$ be a continuous function
and let $\mathfrak a(\cdot, u)$ be increasing on $[-1, 0]$ and decreasing on $[0, 1]$ for all $u \ge 0$.
Then for any function $u \in \W(-1, 1)$, $u \ge 0$,
there exists a sequence $\{u_k\} \subset Lip[-1, 1]$, such that the relation $(\ref{convergence})$ holds.
\end{lm}

\begin{proof}
We can assume that $I( \mathfrak a, u ) < \infty$.

We prove the assertion for the functional
$$I_1( u ) = \int\limits_0^1 F\big( u(x), \mathfrak a(x, u(x)) |u'(x)| \big) \, dx,$$
and the integral over $[-1, 0]$ can be reduced to $I_1$ by changing variable.

We modify the scheme from \cite[Theorem 2.4]{ASC}.
A part of the proof overlaps with \cite{ASC}, but we present a complete proof here for the reader's convenience.

We need the following auxiliary assertion.

\begin{prop}
\label{convToOne}
{\rm \cite[Lemma 2.7]{ASC}.}
Let $\phi_h: [-1, 1] \to \Real_+$ be a sequence of Lipschitz functions satisfying the conditions:
$\phi_h' \ge 1$ for almost every $x$ and all $h$, $\phi_h( x ) \to x$ for almost every $x$.
Then for any $f \in L_1(\Real)$ we have $f(\phi_h) \to f$ in $L_1(\Real)$.
\end{prop}

For $h \in \Nat$ we cover the set $\{ x \in [0, 1]: |u'(x)| > h \}$ with an open set $A_h$.
Without loss of generality, we can assume that
$A_{h + 1} \subset A_{h}$ and $\abs{A_h} \to 0$ for $h \to \infty$.

Denote by $v_h$ the nonnegative continuous function on $[0, 1]$,
coinciding with $u$ on $[0, 1] \setminus A_h$ and
linear on intervals forming $A_h$.
Then $v_h \to u$ in $\W$.
Now we modify $v_h$ to get Lipschitz functions.

Let $A_h = \cup_k \Omega_{h,k}$, where $\Omega_{h,k} = ( b_{h,k}^-, b_{h,k}^+ )$.
Denote
$$\alpha_{h,k} := \abs{\Omega_{h,k}}, \quad
\beta_{h,k} := v_h(b_{h,k}^+) - v_h(b_{h,k}^-) = u(b_{h,k}^+) - u(b_{h,k}^-).$$
Then $v'_h = \frac{\beta_{h,k}}{\alpha_{h,k}}$ in $\Omega_{h,k}$.
Note that
$$\sum_k \abs{\beta_{h,k}} \le \int\limits_{A_h} \abs{u'} \, dx \le \norm{u'}_{L_1(-1, 1)}< \infty,$$
and hence
$\sum_k \abs{\beta_{h,k}} \to 0$ as $h \to 0$ by the Lebesgue theorem.

We define the function $\phi_h \in \W(0, 1)$ as follows:
$$
\begin{aligned}
\phi_h( 0 ) &= 0 & & \\
\phi_h' &=  1 & \text{ in } & [0, 1] \setminus A_h,\\
\phi_h' &=  \max \Big( \frac{ \abs{\beta_{h,k}} }{ \alpha_{h,k} }, 1 \Big) & \text{ in } & \Omega_{h,k}.
\end{aligned}
$$	

Note that $\int_0^1 \abs{\phi_h'} \, dx \le 1 + \sum_k \abs{\beta_{h,k}} < \infty$.

Next, $\phi_h' \to 1$ in $L_1(0, 1)$:
$$\int \abs{\phi_h' - 1} \, dx = \sum\limits_k \Big( \max \Big( \frac{\abs{\beta_{h,k}}}{\alpha_{h,k}}, 1 \Big) - 1 \Big) \alpha_{h,k} \le
\sum\limits_k \abs{\beta_{h,k}} \to 0.$$
Thus $\phi_h$ satisfies the conditions of Proposition \ref{convToOne}.

Consider now $\phi_h^{-1}: [0, 1] \to [0, 1]$ --- the restriction to $[0, 1]$ of the inverse to $\phi_h$.
Then
$$
\begin{aligned}
\phi_h^{-1} ( 0 ) &= 0 & & \\
( \phi_h^{-1} )' &=  1 & \text{ in } & [0, 1] \setminus \phi_h( A_h ),\\
( \phi_h^{-1} )' &=  \min \Big( \frac{ \alpha_{h,k} }{ \abs{ \beta_{h,k} } }, 1 \Big) & \text{ in } & [0, 1] \cap \phi_h( \Omega_{h,k} ).
\end{aligned}
$$

Let $u_h = v_h( \phi_h^{-1} )$.
Note that $u_h(0) = u(0)$, and
\begin{align*}
u_h' &=  v_h'( \phi_h^{-1} ) \cdot ( \phi_h^{-1} )' = u'( \phi_h^{-1} ) & \text{ in } & [0, 1] \setminus \phi_h( A_h ),\\
u_h' &=  v_h'( \phi_h^{-1} ) \cdot ( \phi_h^{-1} )' = 
\sign{ \beta_{h,k} } \cdot \min \Big( 1, \frac{ \abs{ \beta_{h,k} } }{ \alpha_{h,k} } \Big) & \text{ in } & [0, 1] \cap \phi_h( \Omega_{h,k} ).
\end{align*}
Thus, $u_h$ is Lipschitz since $u'$ is bounded in $[0, 1] \setminus A_h$.

We claim that $u_h \to u$ in $\W(0, 1)$. Indeed, it is sufficient to estimate

$$\norm{u_h' - u'}_{L_1} \le \int\limits_{[0, 1] \setminus \phi_h(A_h)} \abs{u_h' - u'} + 
\int\limits_{[0, 1] \cap \phi_h(A_h)} \abs{u_h'} + \int\limits_{[0, 1] \cap \phi_h(A_h)} \abs{u'} =: P_h^1 + P_h^2 + P_h^3.$$
$$P_h^1 = \int\limits_{[0, 1] \setminus \phi_h( A_h )} \abs{u'( \phi_h^{-1} ) - u'} \, dx =
\int\limits_{\phi_h^{-1} ( [0, 1] ) \setminus A_h} \abs{u' - u'( \phi_h )} \, dz \le
\int\limits_{[0, 1]} \abs{u' - u'( \phi_h )} \, dz.$$
By Proposition \ref{convToOne}, $P_h^1 \to 0$.
Further,
$$P_h^2 \le \abs{\phi_h( A_h )} = \sum\limits_k \abs{\phi_h( \Omega_{h,k} )} = \sum\limits_k \max (\abs{\beta_{h,k}}, \alpha_{h,k})
\le \sum\limits_k \alpha_{h,k} + \sum\limits_k \abs{\beta_{h,k}} \to 0.$$
Finally, $P_h^3 \to 0$ by the absolute continuity of the integral, and the assertion is proved.

It remains to show that $I_1( u_h ) \to I_1( u )$.

\begin{multline*}
I_1( u_h ) = \int\limits_{[0, 1] \setminus \phi_h( A_h )} F\big( u_h( x ), \mathfrak a( x, u_h(x) ) |u_h'( x )| \big) \, dx +\\
\int\limits_{[0, 1] \cap \phi_h( A_h )} F\big( u_h( x ), \mathfrak a( x, u_h(x) ) |u_h'( x )| \big) \, dx =: \hat{P_h^1} + \hat{P_h^2}.
\end{multline*}
Since $u \in \W(0, 1)$ then $u \in L_\infty( [0, 1] )$.
Denote $\norm{u}_\infty = r$.
Then $\norm{u_h}_\infty < 2r$ for sufficiently large $h$.
Also, $\abs{u_h'} \le 1$ almost everywhere in $\phi_h( A_h )$.
Then $\hat{P_h^2} \le M_F \abs{\phi_h( A_h )} \to 0$, where
$$M_F = \max\limits_{[-2r, 2r] \times [-M_{\mathfrak a}, M_{\mathfrak a}]} F;\quad M_{\mathfrak a} = \max\limits_{[0, 1] \times [-2r, 2r]} \mathfrak a.$$

Further,
\begin{multline*}
\hat{P_h^1} = \int\limits_{ [0, 1] \setminus \phi_h( A_h ) }
	F\big( u( \phi_h^{-1}( x ) ), \mathfrak a( x, u( \phi_h^{-1}( x ) ) |u'( \phi_h^{-1}( x ) ) ( \phi_h^{-1} )'| ) \big) \, dx
\\ =\int\limits_{ \phi_h^{-1}( [0, 1] ) \setminus A_h } F\big( u( z ), \mathfrak a( \phi_h( z ), u( z ) ) |u'( z )| \big) \, dz
\\ = \int\limits_{ [0, 1] } F\big( u( z ), \mathfrak a( \phi_h( z ), u( z ) ) |u'( z )| \big) \chi_{ \phi_h^{-1}( [0, 1] ) \setminus A_h } \, dz.
\end{multline*}
The last equality, generally speaking, does not make sense, since $\phi_h( z )$ can take values outside $[0, 1]$.
Let us define $\mathfrak a( z, u ) = \mathfrak a( 1, u )$ for $z > 1$. Now the expression is correct.
Note that $\chi_{\phi_h^{-1}( [0, 1] ) \setminus A_h}$ increases,
since sets $\phi_h^{-1}( [0, 1] )$ increase and sets $A_h$ decay,
that is $\phi_{h_1}^{-1}( [0, 1] ) \subset \phi_{h_2}^{-1}( [0, 1] )$ and $A_{h_1} \supset A_{h_2}$ for $h_1 \le h_2$.
Since $\mathfrak a$ is decreasing on $[0, 1]$ (in fact, on $\phi_h( [0, 1] )$) and $\phi_h( z )$ is decreasing in $h$,
then $\mathfrak a( \phi_h( z ) )$ is increasing in $h$.
We apply the monotone convergence theorem and get
$$\hat{P_h^1} \to \int\limits_{[0, 1]} F\big( u( z ), \mathfrak a( z, u( z ) ) |u'( z )| \big) \, dz.$$

\end{proof}

\begin{rem}
Obviously, the proof works for any interval $[x_0, x_1]$ with function $u$ pinned at $x_0$,
provided the weight $\mathfrak a$ is decreasing in $x$ on $[x_0, x_1]$.
That is there exists $\{u_h\}$, such that
\begin{gather*}
u_h(x_0) = u(x_0); \qquad u_h \to u \text{ in } \W(x_0, x_1);\\
\int\limits_{x_0}^{x_1} F\big( u_h(x), \mathfrak a(x, u_h(x)) \abs{u_h'(x)} \big) \to \int\limits_{x_0}^{x_1} F\big( u(x), \mathfrak a(x, u(x)) \abs{u'(x)} \big).
\end{gather*}
Similarly, if $\mathfrak a$ is increasing in $x$, the same works for functions $u$ pinned at the right end of the segment.
\end{rem}

\begin{cor}
Suppose that the function $\mathfrak a$ is continuous, even in $x$, decreasing on $[0, 1]$ and satisfies $(\ref{almostConcave})$.
Then for every $u \in \W(-1, 1)$ the inequality (\ref{toprove}) holds.
\end{cor}

\begin{proof}
The statement follows from Lemmata \ref{uplift} and \ref{Wapprox} immediately.
\end{proof}

\section{The result in the general case}
\label{moveForth}

Now we want to get rid of the monotonicity restriction on the weight.
We do this in several steps.

To begin, we note that all properties of the function $\mathfrak a$ are of interest
only in the neighborhood of the graphs of functions $u$ and $u^*$.

We introduce the following conditions each of which, being added to the previous ones, defines a smaller class of weight functions:

\bigskip

\smallskip
\noindent
$(H1)$ $\mathfrak a(x, v)$ satisfies (\ref{almostConcave}), is even in $x$ and $I(\mathfrak a, u) < \infty$.
\smallskip

\bigskip
\noindent
$(H2)$ the number of zeros of $\mathfrak a(\cdot, v)$ is bounded by a constant independent of $v$
for all $v \in [\min u(x), \max u(x)]$ such that $\mathfrak a(\cdot, v) \not \equiv 0$.

\bigskip
\noindent
$(H3)$ If $\mathfrak a(x_0, u(x_0)) = 0$ for some $x_0$, then $\mathfrak a(\cdot, u(x_0)) \equiv 0$.
Moreover, $\lim\limits_{k \to \infty} D_k(\mathfrak a, U(\mathfrak a)) = 0$, where
$$U(\mathfrak a) := \{ v \in [\min u(x), \max u(x)]: \mathfrak a(\cdot, v) \not \equiv 0 \},$$
\begin{equation}
\label{bigD}
D_k(\mathfrak a, U): = \sup \limits_{v \in U}
\frac{\max \limits_{\abs{x_1 - x_2} \le \frac{2}{k}} \abs{\mathfrak a(x_1, v) - \mathfrak a(x_2, v)}}
{\min \limits_{\dist (x, u^{-1} (v)) \le \frac{2}{k}} \mathfrak a(x, v)}.
\end{equation}

\bigskip
\noindent
$(H4)$ There exists an even $k$, such that $\mathfrak a(\cdot, v)$ are linear for each $v$ on each of the segments
$[-1 + \frac{2i}{k}, -1 + \frac{2(i + 1)}{k}]$.

\bigskip
\noindent
$(H5)$ The difference between the
set of $v \in \Real_+$, for which $\mathfrak a(\cdot, v)$ has segments of constant values,
and the set of $v \in \Real_+$ such that $\mathfrak a(\cdot, v) \equiv 0$
has zero measure.

\bigskip
\noindent
$(H6)$ The segment $[-1, 1]$ can be represented as a unity of touching segments
on each of which $\mathfrak a$ does not change the monotonicity with respect to $x$ in a $v$-neighborhood of the graph of the function $u$.

\bigskip
\noindent
$(H7)$ Let $x_1 < x_2 < x_3$,
let $\mathfrak a(\cdot, v)$ decrease for $x \in [x_1, x_2]$ in a $v$-neighborhood of the graph of the function $u$,
and let $\mathfrak a(\cdot, v)$ increase for $x \in [x_1, x_2]$ in a $v$-neighborhood of the graph of the function $u$.
Then we have $\mathfrak a(\cdot, v) \equiv 0$ in a $v$-neighborhood of $u(x_2)$.

\bigskip

The weights satisfying $(H1)$ will be called {\it admissible for a given $u$}.

\medskip

Now we can formulate the main assertion of our work.
\begin{thm}
\label{mainThm}
Suppose $F \in \mathfrak{F}$, the function $u \in \W(-1, 1)$ is non-negative,
and the weight function $\mathfrak a: [-1, 1] \times \Real_+ \to \Real_+$ is continuous
and admissible for $u$.
Then the inequality $(\ref{toprove})$ holds.
\end{thm}

We prove the inequality (\ref{toprove}) under conditions $(H1)-(H7)$,
and then get rid of extra conditions one by one.

For the proof we need the following facts.

\begin{prop}
\label{levelDerivative}
{\rm \cite[Theorem 6.19]{LL} }
For every $u \in \W(-1, 1)$ and for an arbitrary set $A \subset \Real_+$ of zero measure,
$u'(x) = 0$ almost everywhere in $u^{-1}(A)$.
\end{prop}

\begin{lm}
\label{zeroApprox}
Suppose that $u \in \W(-1, 1)$ is nonnegative.
Let a closed set $W \subset \Real_+$ be such that
the set of $v \in W$, for which $\mathfrak a(\cdot, v) \not\equiv 0$, has zero measure.
Then there exists an increasing sequence of weights $\mathfrak b_{\ell}$, which satisfy

1) $\mathfrak b_{\ell}(\cdot, v) \rightrightarrows \mathfrak a(\cdot, v)$ for almost all $v$;

2) $\mathfrak b_{\ell}(\cdot, v) \equiv 0$ for every $v$ in some neighborhood of $W$ (the neighborhood depends on $\ell$);

3) $I(\mathfrak b_{\ell}, u) \to I(\mathfrak a, u)$ and $I(\mathfrak b_{\ell}, u^*) \to I(\mathfrak a, u^*)$.
\end{lm}

\begin{rem}
If $a$ is admissible for $u$ then $b_{\ell}$ are also admissible.
\end{rem}

\begin{proof}
Take $\rho(d) := \min(1, \max(0, d))$,
$$\mathfrak b_{\ell}(x, v) := \mathfrak a(x, v) \cdot \rho(\ell \dist(v, W) - 1) \le \mathfrak a(x, v).$$
This weight is equal to zero in $\left(\frac{1}{\ell}\right)$-neighborhood of $W$.
In addition, $\mathfrak b_{\ell} \equiv \mathfrak a$ outside the $\left(\frac{2}{\ell}\right)$-neighborhood of $W$ and
$\mathfrak b_{\ell}(x, v)$ increases in $\ell$.
Thus, $\mathfrak b_{\ell}(\cdot, v) \rightrightarrows \mathfrak a(\cdot, v)$ for almost all $v$.
By the monotone convergence theorem
$I(u^{-1}(\Real_+ \setminus W), \mathfrak b_{\ell}, u) \nearrow I(u^{-1}(\Real_+ \setminus W), \mathfrak a, u)$.

Divide the set $W$ into $W_1 := \{v \in W: \mathfrak a(\cdot, v) \equiv 0\}$ and $W_2 = W \setminus W_1$.
Then
$$
\begin{aligned}
I(u^{-1}(W_1), \mathfrak b_{\ell}, u) &= I(u^{-1}(W_1), \mathfrak a, u),\\
I(u^{-1}(W_2), \mathfrak b_{\ell}, u) &= \int\limits_{x \in u^{-1}(W_2)} F\big(u(x), \mathfrak b_{\ell}(x, u(x)) |u'(x)|\big) \, dx.
\end{aligned}
$$
By Proposition \ref{levelDerivative}, $u'(x) = 0$ almost everywhere on $u^{-1}(W_2)$.
Thus
$$I(u^{-1}(W_2), \mathfrak b_{\ell}, u) = \int\limits_{x \in u^{-1}(W_2)} F\big(u(x), 0\big) \, dx = 0.$$
Similarly, $I(u^{-1}(W_2), \mathfrak a, u) = 0$. Hence $I(\mathfrak b_{\ell}, u) \to I(\mathfrak a, u)$.
The second relation in {\it 3)} is proved by the same arguments.
\end{proof}

We proceed to the proof of the theorem.

\bigskip
{\bf Step 1.} {\it Let $u \in \W(-1, 1)$ and let the weight $\mathfrak a$ satisfy the conditions $(H1)-(H7)$.
Then the inequality (\ref{toprove}) holds.}

Divide the segment $[-1, 1]$ into touching subsegments $\Delta_j$, each consisting of two parts.
On the left part of each $\Delta_j$ the weight $\mathfrak a$ increases in $x$ in a neighborhood
of the graph of $u(x)$. On the right part it decreases.
On each $\Delta_j$ we can apply the construction from the previous section
for approximating $u$ with Lipschitz functions $u_n$.
This gives us $I(\Delta_j, \mathfrak a, u_n) \to I(\Delta_j, \mathfrak a, u)$.

However, approximating functions $u_n$ have discontinuities at the borders of the segments $\Delta_j$
(denote them by $\hat{x}_j$).

Note that according to the condition $(H7)$ one can choose points $\hat{x}_j$ so
that $\mathfrak a \equiv 0$ in $(x, v)$-neighborhoods of the points $(\hat{x}_j, u(\hat{x}_j))$.

Next, substitute functions $u_n$ in these neighborhoods of $\hat{x}_j$ with linear pieces
making $u_n$ continuous on $[-1, 1]$.
In view of the above, this does not change the integrals $I(\Delta_j, \mathfrak a, u_n)$,
and we get $I(\mathfrak a, u_n) \to I(\mathfrak a, u)$.

By Lemma \ref{uplift} we obtain (\ref{toprove}).

\bigskip

{\bf Step 2.} {\it Let the weight $\mathfrak a$ satisfy the conditions $(H1)-(H6)$.
Then the inequality (\ref{toprove}) holds.}

We apply Lemma \ref{zeroApprox} with the following set $W$:
the set of all $v$, at which the graph of $u(x)$ traverses from a rectangle,
in which the weight decreases in $x$,
to a rectangle in which the weight increases.
Obviously, the resulting function $\mathfrak b_{\ell}$ satisfy $(H1)-(H7)$.
By Step 1, $I(\mathfrak b_{\ell}, u^*) \le I(\mathfrak b_{\ell}, u)$.
Passing to the limit, we obtain (\ref{toprove}).

\bigskip

{\bf Step 3.} {\it Let the weight $\mathfrak a$ satisfy the conditions $(H1)-(H5)$.
Then the inequality (\ref{toprove}) holds.}

Consider abscissas of nodes of $\mathfrak a$
and ordinates, for which $\mathfrak a$ has constant pieces.
They define a division of the rectangle $[-1, 1] \times [\min u(x), \max u(x)]$
into rectangles in each of which the weight $\mathfrak a$ is monotone in $x$.
However, the number of rectangles can be infinite.
Also, if the graph of $u$ crosses a horizontal boundary of some rectangle,
monotonicity in the $v$-neighborhood of the point of intersection may change.

Consider set $W$ containing all $v$, for which the weight $\mathfrak a$ has constant pieces.
Due to $(H5)$ the set of all  $v \in W$ such that $a(\cdot, v) \not\equiv 0$ has zero measure.

We apply Lemma \ref{zeroApprox} and obtain a sequence of weights $\mathfrak b_{\ell}$.
We claim that each of them has only finite number of monotonicity rectangles.
Indeed, any two vertically adjacent rectangles with different monotonicity
are separated by a stripe of $\frac{2}{\ell}$ width with zero values.

The weight $b_{\ell}$ can change monotonicity along the graph of $u$
either at the points $x = -1 + \frac{2 i}{k}$ or where the graph crosses a stripe of zero values.
Note that only finite number of such crossings can arise since
$\int |u'|$ gains at least $\frac{2}{\ell}$ at any crossing and $u' \in L_1(-1, 1)$.

Thereby, $\mathfrak b_{\ell}$ satisfy $(H1)-(H6)$. By Step 2, $I(\mathfrak b_{\ell}, u^*) \le I(\mathfrak b_{\ell}, u)$.
Passing to the limit, we obtain (\ref{toprove}).

\bigskip
{\bf Step 4.} {\it Let the weight $\mathfrak a$ satisfy the conditions $(H1)-(H3)$.
Then the inequality (\ref{toprove}) holds.}

Suppose that the function $\mathfrak a$ satisfies $(H1)-(H3)$, in particular $I(\mathfrak a, u) < \infty$.

We fix an arbitrary even $k$.
For each $v$ we interpolate $\mathfrak a$ with piecewise linear functions
with nodes $( -1 + \frac{2i}{k}, \mathfrak a(-1 + \frac{2i}{k}, v) )$.
Resulting function $\mathfrak a_k(x, v)$ is continuous, even in $x$
and satisfies (\ref{almostConcave}) by Lemma \ref{piecewiseLinearConcave}.
In addition, $\mathfrak a_k \to \mathfrak a$ when $k \to \infty$,
moreover the convergence is uniform on compact sets.
However, the inequality $\mathfrak a_k(x, u(x)) \le \mathfrak a(x, u(x))$ can be violated,
and thus $\mathfrak a_k$ may be non-admissible for $u$.

Set $\mathfrak c_k := (1 - D_k(\mathfrak a_k, U(\mathfrak a_k))) \mathfrak a_k$, where $D_k$ is defined in (\ref{bigD}).
$D_k(\mathfrak a_k, U(\mathfrak a_k))$ are positive and tend to zero, thus $\mathfrak c_k \to \mathfrak a$ while $k \to \infty$.
We claim that $\mathfrak c_k(x, u(x)) \le \mathfrak a(x, u(x))$.

Indeed, consider some
$x \in [-1 + \frac{2i}{k}, -1 + \frac{2(i + 1)}{k}] =: [x_i, x_{i + 1}]$.
Then $\mathfrak c_k(x, u(x)) \le \max( \mathfrak c_k(x_i, u(x)), \mathfrak c_k(x_{i + 1}, u(x)) )$, because
$\mathfrak c_k$ is piecewise linear in $x$. Moreover,
\begin{multline*}
\mathfrak c_k(x_i, u(x)) = ( 1 - D_k(\mathfrak a_k, U(\mathfrak a_k))) \cdot \mathfrak a(x_i, u(x)) \\
\le \mathfrak a(x_i, u(x)) - \frac{\mathfrak a(x_i, u(x)) - \mathfrak a(x, u(x))}{\mathfrak a(x_i, u(x))} \cdot \mathfrak a(x_i, u(x)) = \mathfrak a(x, u(x)).
\end{multline*}
Similarly $\mathfrak c_k(x_{i + 1}, u(x)) \le \mathfrak a(x, u(x))$.
Thus, $\mathfrak c_k(x, u(x)) \le \mathfrak a(x, u(x))$ for any $x$, and $\mathfrak c_k$ are admissible for $u$.
Thereby the functions $\mathfrak c_k$ satisfy $(H1)-(H4)$.

For a given $k \in \Nat$, we approximate the function $\mathfrak c_k =: \mathfrak c$ with weights satisfying $(H1)-(H5)$.
Consider the auxiliary function $\Lambda(x) = 1 - \abs{x}$, satisfying (\ref{almostConcave}).

Take
$$t(v):=D_k(\mathfrak c, U(\mathfrak c)) \cdot \max\{\tau \ge 0: \forall x \in u^{-1}(v) \quad \tau \Lambda(x) \le \mathfrak c(x, u(x))\}.$$
The function $t$ depends on $k$, but we omit this fact in presentation.

It is clear that the maximum $\tau$ is zero only if $\mathfrak c(\cdot, v) \equiv 0$,
since otherwise the condition $(H3)$ is violated.

Function $t$ may be discontinuous. However, it is easy to see that it is lower semicontinuous.
Next, we take
$$\tilde{t}(v) := \inf_{w \in u([-1, 1])} \{t(w) + |v - w|\}.$$
It is obvious that $\tilde{t} \le t$, and the set of zeros of $t$ and $\tilde{t}$ coincide.

We claim that $\tilde{t}$ is continuous (and even Lipschitz).
Indeed, take some $v_1$.
Then there is an arbitrarily small $\eps > 0$ and $w_1 \in u([-1, 1])$
satisfying $\tilde{t}(v_1) = t(w_1) + |v_1 - w_1| - \eps$.
For every $v_2$, we have $\tilde{t}(v_2) \le t(w_1) + |v_2 - w_1|$.
And thus $\tilde{t}(v_2) - \tilde{t}(v_1) \le |v_1 - v_2| + \eps$.
By the arbitrariness of $v_1$, $v_2$ and $\eps$, the claim follows.

For $\alpha \in [0, 1]$ the function $\mathfrak d_\alpha(x, v) := \mathfrak c(x, v) + \alpha \Lambda(x) \tilde{t}(v)$
is even in $x$, satisfies (\ref{almostConcave}) in concordance with Lemma \ref{maxSumConcave},
and does not exceed $\mathfrak a(x, v)$ due to the construction of the function $\tilde{t}$.
Thus, $\mathfrak d_\alpha$ is an admissible weight.
Also, it is obvious that $\mathfrak d_\alpha$ satisfies $(H1)-(H4)$.

Let us show that there exists a sequence $\alpha_j \searrow 0$
such that $\mathfrak d_{\alpha_j}(\cdot, v)$ has no segments of constant values,
unless $\mathfrak d_{\alpha_j}(\cdot, v) \equiv 0$ or $v$ belongs to a zero measure set.
We introduce the set of $\alpha$, which are ``bad'' on $[x_i, x_{i + 1}]$:
\begin{multline*}
A_i := \big \{\alpha \in [0, 1]: \\
meas \{v \in [\min u, \max u]: \frac{\mathfrak c(x_{i + 1}, v) - \mathfrak c(x_i, v))}{\frac{2}{k}} + \alpha \chi_i \tilde{t} (v) = 0 \} > 0 \big \},
\end{multline*}
where $\chi_i = 1$ if $[x_i, x_{i + 1}] \subset [0, 1]$, and $\chi_i = -1$ if $[x_i, x_{i + 1}] \subset [-1, 0]$.

Consider the following function
$$
\begin{aligned}
h_i(v) = & \frac{\mathfrak c(x_{i + 1}, v) - \mathfrak c(x_i, v)}{\tilde{t} (v)} & \text{ if } \tilde{t} (v) \neq 0 & \\
h_i(v) = & 0 & \text{ if } \tilde{t} (v) = 0 &.
\end{aligned}
$$
We have $\card(A_i) = \card(\{ \alpha \in [0, 1]: meas \{ v \in [\min u, \max u]: h_i(v) \pm \frac{2}{k} \alpha = 0 \} > 0 \}).$
Then $\card(A_i) \le \aleph_0$, and $\card(\cup_i A_i) \le \aleph_0$.
Thus, there exists a sequence of weights $\mathfrak d_{\alpha_j} \searrow \mathfrak c$, satisfying $(H1)-(H5)$.
By Step 3, $I(\mathfrak d_{\alpha_j}, u^*) \le I(\mathfrak d_{\alpha_j}, u)$.
Passing to the limit, we get $I(\mathfrak c, u^*) \le I(\mathfrak c, u)$.

Further, for $x \in [-1, 1]$ we have
\begin{equation}
\label{step4Conv}
F\big(u(x), \mathfrak c_k(x, u(x)) |u'(x)|\big) \to F\big(u(x), \mathfrak a(x, u(x)) |u'(x)|\big)
\end{equation}
as $k \to \infty$.
Moreover, $F\big(u(x), \mathfrak a(x, u(x)) |u'(x)|\big)$ is an integrable majorant
for the left-hand side in (\ref{step4Conv}).
By the Lebesgue theorem, we have $I(\mathfrak c_k, u) \to I(\mathfrak a, u)$.
Since $I(\mathfrak c_k, u^*) \le I(\mathfrak c_k, u)$, Lemma \ref{uplift} proves the inequality (\ref{toprove}).

\bigskip
{\bf Step 5.} {\it Let the weight $\mathfrak a$ satisfy only the condition $(H1)$.
Then the inequality (\ref{toprove}) holds.}

We approximate $\mathfrak a$ by weights satisfying $(H1)-(H2)$.
To do this we apply Lemma \ref{zeroApprox} with $W = \{ v \in \Real_+: \mathfrak a(\cdot, v) \equiv 0 \}$.
Let us introduce the notation $$Z_{\mathfrak a}(v) := \{ x \in [-1, 1]: \mathfrak a(x, v) = 0 \}.$$
Note that the sets $Z_{\mathfrak b_{\ell}}(v)$ are either $Z_{\mathfrak a}(v)$ or $[-1, 1]$.

Let us show that $\mathfrak b_{\ell}$ satisfies $(H2)$.
Indeed, otherwise there is a sequence $v_m$, for which
$m < \card(Z_{\mathfrak b_{\ell}})(v_m) < \infty$.
After passing to a subsequence, we have $v_m \to v_0$.
Part 2 of Lemma \ref{periodicity} implies that the set $Z_{\mathfrak b_{\ell}}(v_m) = Z_{\mathfrak a}(v_m)$
is periodic with period less or equal to $\frac{2}{m - 1}$.
Take some $x \in [-1, 1]$. For each $m$ there exists $x_m$ such that
$\abs{x - x_m} \le \frac{1}{m - 1}$ and $\mathfrak a(x_m, v_m) = 0$.
But $\mathfrak a(x_m, v_m) \to \mathfrak a(x, v_0)$.
Therefore, $\mathfrak a(x, v_0) = 0$.

Thus $Z_{\mathfrak a}(v_0) = [-1, 1]$.
But this means that for every $v$ such that $\abs{v - v_0} \le \frac{1}{\ell}$,
we have $\mathfrak b_{\ell}(\cdot, v) \equiv 0$,
which contradicts $\card(Z_{\mathfrak b_{\ell}})(v_m) < \infty$.

Now we fix $\ell \in \Nat$ and denote $\mathfrak b_{\ell} =: \mathfrak b$.
Let us approximate the function $\mathfrak b$ with weights satisfying $(H1)-(H3)$.
It follows from $(H2)$, that there exists a set $T \subset [-1, 1]$
consisting of a finite number of elements, such that
if $x \not\in T$ and $\mathfrak b(x, v) = 0$ for some $v$, then $\mathfrak b(\cdot, v) \equiv 0$.

We use Lemma \ref{zeroApprox} with $W = u(T) \cup u^*(T)$.
The weights $\mathfrak c_j$, given by the Lemma, satisfy $(H1)-(H2)$,
since they are just $\mathfrak b$ multiplied by a factor less than one, which depends only on $v$.

For any $k$ sufficiently large, there exists $j = j(k)$ such that
$$u\Big( \Big\{ x \in [-1, 1]: dist(x, T) \le \frac{4}{k} \Big\} \Big) \subset \Big\{ v \in \Real_+: dist(v, u(T)) \le \frac{1}{2j} \Big\},$$
and $j(k) \to \infty$ as $k \to \infty$ by continuity of $u$.
This implies that $\min\limits_{dist(x, u^{-1}(v)) \le \frac{2}{k}} c_j(x, v) > 0$
for all $v \in U(c_j)$.
Moreover, for $v \in U(c_j)$ we have
$$
\frac{\max\limits_{\abs{x_i - x_{i + 1}} \le \frac{2}{k}} \abs{\mathfrak c_j(x_i, v) - \mathfrak c_j(x_{i + 1}, v)}}
{\min\limits_{\dist(x, u^{-1}(v)) \le \frac{2}{k}} \mathfrak c_j(x, v)}
=\frac{\max\limits_{\abs{x_i - x_{i + 1}} \le \frac{2}{k}} \abs{\mathfrak b(x_i, v) - \mathfrak b(x_{i + 1}, v)}}
{\min\limits_{\dist(x, u^{-1}(v)) \le \frac{2}{k}} \mathfrak b(x, v)}.
$$
Note, that the denominator of the right-hand side is separated from zero for $v \in U(\mathfrak c_j)$.
Thus, $D_k(\mathfrak c_j, U(\mathfrak c_j))$ is bounded.

Since $D_k$ does not change if we multiply the first argument by a positive factor independent of $x$,
and $U(\mathfrak c_j) \nearrow U(\mathfrak b)$, we have
$$D_k(\mathfrak c_j, U(\mathfrak c_j)) = D_k(\mathfrak b, U(\mathfrak c_j)) \le D_k(\mathfrak b, U(\mathfrak b)) \to 0$$
as $k \to \infty$.

Thus, the weights $\mathfrak c_{j(k)}$ satisfy $(H1)-(H3)$.
By Step 4, $I(\mathfrak c_{j(k)}, u^*) \le I(\mathfrak c_{j(k)}, u)$.
Passing to the limit, we get $I(\mathfrak b_{\ell}, u^*) \le I(\mathfrak b_{\ell}, u)$,
and consequently the inequality (\ref{toprove}).

Thus, Theorem \ref{mainThm} is proved.
\hfill $\square$

\medskip

Now we consider the case where the function $u$ satisfies the additional condition $u(-1) = 0$.
\begin{thm}
Suppose that $F \in \mathfrak{F}$, the function $u \in \W(-1, 1)$ is nonnegative, $u(-1) = 0$,
and the weight function $\mathfrak a: [-1, 1] \times \Real_+ \to \Real_+$ is continuous
and satisfies $(\ref{almostConcave})$.
Then the inequality $(\ref{toprove})$ holds.
\end{thm}

\begin{proof}
We follow the proof of Theorem \ref{mainThm},
but we change $(H1)$ and $(H7)$ to the following conditions:

\bigskip
\noindent
$(H1')$ $\mathfrak a(x, v)$ satisfies (\ref{almostConcave}), and $I(\mathfrak a, u) < \infty$.

\bigskip
\noindent
$(H7')$ The assumption $(H7)$ is satisfied and $\mathfrak a(\cdot, v) \equiv 0$ in some $v$-neighborhood of zero.

\bigskip
{\bf Step 1.} {\it Let $u \in \W(-1, 1)$, $u(-1) = 0$ and let the weight $\mathfrak a$ satisfy the conditions $(H1'), (H2)-(H6), (H7')$.
Then the inequality (\ref{toprove}) holds.}

To prove this we approximate the function $u$ in the same way as in the first step of Theorem \ref{mainThm} proof,
changing $u$ in a neighborhood of $x = -1$ to a linear function with $u_n(-1) = 0$ preserved.

\bigskip
{\bf Step 2.} {\it Let the weight $\mathfrak a$ satisfy conditions $(H1'), (H2)-(H6)$.
Then the inequality (\ref{toprove}) holds.}

To prove this we add zero to the set $W$ from the second step of Theorem \ref{mainThm} proof,
and repeat the rest of the proof.

\medskip

Further steps are unchanged.
\end{proof}

\section{Appendix. The case of symmetric rearrangement}

\subsection{Necessary conditions for the weight}

\begin{lm}
If the inequality $(\ref{toproveSymm})$ holds for all $F \in \mathfrak{F}$
and all piecewise linear $u$, then
the weight $\mathfrak a$ satisfies
\begin{equation}
\label{almostConvex}
\forall s, t \in [-1, 1], \forall v \in \Real_+ \quad
\mathfrak a( s, v ) + \mathfrak a( t, v ) \ge \mathfrak a\Big( \frac{ s - t }{2}, v \Big) + \mathfrak a\Big( \frac{ t - s }{2}, v \Big).
\end{equation}
\end{lm}

\begin{proof}
Assume that the inequality (\ref{almostConvex}) is not satisfied.
Then there are $-1 \le s < t \le 1$, $\eps, \delta > 0$ ($2 \eps < t - s$) and $\bar{v} \in \Real_+$,
such that for any $0 \le z \le \eps$ and any $\bar{v} \le v \le \bar{v} + \eps$ the following holds:
\begin{equation}
\label{notConvex}
\mathfrak a(s + z, v + z) + \mathfrak a(t - z, v + z) + 2 \delta < \mathfrak a\Big(\frac{s - t}{2} + z, v + z \Big) + \mathfrak a\Big(\frac{t - s}{2} - z, v + z \Big).
\end{equation}

Consider the function $u$ defined in (\ref{parLinU}). We have
$$
\left\{
\begin{aligned}
\bar{u}(x) &= \bar{v}, & x \in &[-1, \frac{s - t}{2}] \cup [\frac{t - s}{2}, 1]\\
\bar{u}(x) &= \bar{v} + x - \frac{s - t}{2}, & x \in &[\frac{s - t}{2}, \frac{s - t}{2} + \eps]\\
\bar{u}(x) &= \bar{v} + \eps, & x \in &[\frac{s - t}{2} + \eps, \frac{t - s}{2} - \eps]\\
\bar{u}(x) &= \bar{v} + \frac{t - s}{2} - x, & x \in &[\frac{t - s}{2} - \eps, \frac{t - s}{2}].
\end{aligned}
\right.
$$

Hence we obtain
\begin{multline*}
0 \le I(\mathfrak a, u) - I(\mathfrak a, \overline{u}) \\
=\int_0^{\eps} F\big( u(s + z), \frac{\mathfrak a( s + z, u(s + z) )}{\eps} \big) dz + \int_0^{\eps} F\big( u(t - z), \frac{\mathfrak a(t - z, u(t - z))}{\eps} \big) dz \\
-\int_0^{\eps} F\big( \bar{u}(\frac{s - t}{2} + z), \frac{\mathfrak a( \frac{s - t}{2} + z, \bar{u}(\frac{s - t}{2} + z) )}{\eps} \big) dz \\
-\int_0^{\eps} F\big( \bar{u}(\frac{t - s}{2} - z), \frac{\mathfrak a( \frac{t - s}{2} - z, \bar{u}(\frac{t - s}{2} - z) )}{\eps} \big) dz =: J.
\end{multline*}

Take $F(v, p) := f(p) := p + \gamma p^2$, where $\gamma > 0$.
Then
\begin{multline*}
J = \int_0^{\eps} \big( f(\frac{\mathfrak a(s + z, \bar{v} + z)}{\eps}) + f(\frac{\mathfrak a(t - z, \bar{v} + z)}{\eps}) \\
- f(\frac{\mathfrak a(\frac{s - t}{2} + z, \bar{v} + z)}{\eps}) - f(\frac{\mathfrak a(\frac{t - s}{2} - z, \bar{v} + z)}{\eps}) \big) dz.
\end{multline*}

We define $A$ by relation (\ref{weightMax}).
If we take $\gamma := \frac{\delta / \eps}{(A / \eps)^2} > 0$,
then for $p \le \frac{A}{\eps}$ we have $p \le f( p ) \le p + \frac{\delta}{\eps}$, and
\begin{equation*}
J \le \frac{1}{\eps} \int_0^{\eps} \big( \mathfrak a(s + z, \bar{v} + z) + \mathfrak a(t - z, \bar{v} + z) + 2 \delta
- \mathfrak a(\frac{s - t}{2} + z, \bar{v} + z) - \mathfrak a(\frac{t - s}{2} - z, \bar{v} + z) \big) dz < 0
\end{equation*}
(the last inequality follows from (\ref{notConvex})).

Thus, we get a contradiction, hence (\ref{almostConvex}) holds.
\end{proof}

\begin{lm}
Let relation $(\ref{almostConvex})$ hold for a function $\mathfrak a \in C([-1, 1] \times \Real_+)$.
Then $\mathfrak a$ is even and convex with respect to the first argument.
\end{lm}

\begin{proof}
Assume first that $\mathfrak a(\cdot, v) \in C^1([-1, 1])$ for each $v$.
We fix arbitrary $s \in [-1, 1]$ and $v \in \Real_+$ and consider the function
$$b(x) := \mathfrak a( s, v ) + \mathfrak a( x, v ) - \mathfrak a( \frac{ s - x }{2}, v ) - \mathfrak a( \frac{ x - s }{2}, v ) \ge 0.$$
$x = -s$ is the minimum point of $b$, since $b(-s) = 0$.
Hence,
$$b'(-s) = \mathfrak a'_x( -s, v ) + \frac{1}{2} \mathfrak a'_x( s, v ) - \frac{1}{2} \mathfrak a'_x( -s, v ) = 0,$$
that is $\mathfrak a'_x( s, v ) = -\mathfrak a'_x( -s, v )$. Thus, the function $\mathfrak a(\cdot, v)$ is even.

Now consider the case of a continuous $\mathfrak a$.

Define $\mathfrak a( x, v ) := \mathfrak a( -1, v )$ for $x < -1$ and $\mathfrak a( x, v ) := \mathfrak a( 1, v )$ for $x > 1$.
Consider the mollification of the function:
$$\mathfrak a_\rho( x, v ) = \int_\Real \omega_\rho ( z ) \mathfrak a( x - z, v ) dz = \int_\Real \omega_\rho ( z ) \mathfrak a( x + z, v ) dz,$$
where $\omega_\rho(z)$ is a smoothing kernel with radius $\rho$.
Then
\begin{multline*}
\mathfrak a_\rho( s, v ) + \mathfrak a_\rho( t, v ) - \mathfrak a_\rho( \frac{ s - t }{2}, v ) - \mathfrak a_\rho( \frac{ t - s }{2}, v ) =
\\ \int_\Real \omega_\rho ( z ) \big( \mathfrak a( s - z, v ) + \mathfrak a( t + z, v ) - \mathfrak a( \frac{ s - t }{2} - z, v ) - \mathfrak a( \frac{ t - s }{2} + z, v ) \big) dz \ge 0.
\end{multline*}
So $\mathfrak a_\rho(\cdot, v)$ is even.
Passing to the limit with $\rho \to 0$, we obtain that $\mathfrak a(\cdot, v)$ is even.

Finally, for any $s$, $t$ and $v$, we have
$$\mathfrak a( s, v ) + \mathfrak a( t, v ) = \mathfrak a( s, v ) + \mathfrak a( -t, v ) \ge 2 \mathfrak a\big( \frac{ s + t }{2}, v \big).$$
\end{proof}

\subsection{The proof of the inequality (\ref{toproveSymm})}
\label{sobolevSymm}

\begin{thm}
\label{symmThm}
Suppose that $F \in \mathfrak{F}$, the function $u \in \W(-1, 1)$ is non-negative,
and the continuous weight function $\mathfrak a: [-1, 1] \times \Real_+ \to \Real_+$
is even and convex with respect to the first argument.
Then the inequality $(\ref{toproveSymm})$ holds.
\end{thm}

\begin{proof}
As we mentioned in the introduction,
the statement is proved for Lipschitz functions $u$ in paper \cite{Br}.
Thus, we need only to extend it to $\W$-functions.

The case of convex weight is much simpler than the case considered in Section \ref{moveForth}.
Namely, the function $\mathfrak a$ decreases for $x < 0$ and increases for $x > 0$ regardless of $v$.
Thus, the assumption $(H6)$ of Theorem \ref{mainThm} is satisfied.
To fulfil the assumption $(H7)$ we apply Lemma \ref{zeroApprox} with $W = \{ u(0) \}$.
Then we can use immediately Step 1 of the proof of Theorem \ref{mainThm}.
This gives us (\ref{toproveSymm}).
Since Step 1 uses assumptions $(H1)$, $(H6)$, $(H7)$ only,
we do not need to check $(H2)-(H5)$.
\end{proof}

\vskip 40pt

We are grateful to Professor V.G.~Osmolovskii for valuable comments,
which helped to improve the text of the paper.

Authors were supported by RFBR grant 12-01-00439.
The second author was also supported by St. Petersburg University grant 6.38.670.2013.

\end{document}